\documentclass[12pt]{article}
\usepackage[a4paper, total={6in, 8in}, margin={1in}]{geometry}
\usepackage[utf8]{inputenc}

\usepackage{amsmath}
\usepackage{amsthm}
\usepackage{amssymb}
\usepackage{amsfonts}
\usepackage[hidelinks]{hyperref}
\usepackage{scrextend}

\theoremstyle{plain}
\newtheorem{theorem}{Theorem}
\newtheorem{corollary}[theorem]{Corollary}

\theoremstyle{definition}

\theoremstyle{remark}
\newtheorem{remark}[theorem]{Remark}

\newcommand{\nobarfrac}{\genfrac{}{}{0pt}{}}
\newcommand{\seqnum}[1]{\href{https://oeis.org/#1}{\rm \underline{#1}}}

\title{A Study of Second-Order Linear Recurrence Sequences via Continuants}
\author{Hongshen Chua}
\date{}

\begin{document}
\maketitle

\section{Introduction} \label{Sec:Introduction}

Consider the $ d $-periodic Fibonacci numbers, defined by the second-order linear recurrence
\begin{align*}
	f_\nu = a_r f_{\nu - 1} + b_r f_{\nu - 2}, \quad \text{ where } \nu \equiv r \pmod{d},
\end{align*}
with initial values $ f_0 = 0 $ and $ f_1 = 1 $. Such a sequence was studied by Carson \cite{Carson_07}, who, using matrix manipulation, managed to rewrite the sequence in the form:
\begin{align*}
	f_\nu = af_{\nu - d} + bf_{\nu - 2d},
\end{align*}
for some constants $ a $ and $ b $. This extends an earlier partial result by Lehmer \cite{Lehmer_75}. An alternative method via continuants was provided by Panario et al.\ \cite{Panario_Sahin_Wang_13}. In the same paper, they also derived the generating function and Binet formula for $ d $-periodic Fibonacci numbers, supplementing earlier work by Edson et al.\ \cite{Edson_Lewis_Yayenie_11}. Our paper will follow this approach closely. Other related studies include the papers by Tan \cite{Tan_11} and Yayenie \cite{Yayenie_11}, which focus on	the special case of $ d = 2 $. 

The concept of continuants was initially introduced by Euler during the development of a theory for continued fractions. In particular, continuants appear in the numerators and denominators of the convergents to a continued fraction. Subsequently, Muir \cite{Muir_60} studied continuants as determinants of tridiagonal matrices. A major contribution by Muir is the notation that we employ in this paper. Another perspective on continuants is through a combinatorial approach, where they are viewed as a tiling problem, as discussed by Benjamin et al.\ \cite{Benjamin_Su_Quinn_00}. However, the main inspiration for the content of this paper stems from Perron's book \cite{Perron_13} on continued fractions. 	

The goal of this paper is to highlight the similarities between continuants and second-order linear recurrence sequences, with a focus on $ d $-periodic Fibonacci numbers. Additionally, we address a problem posed by Panario et al.\ \cite{Panario_Sahin_Wang_13}, seeking an application for generalized continuants. We hope that this will give continuants the attention they deserve.

While not a prerequisite to understand this paper, we assume here that readers are familiar with Lucas sequences, especially with the Fibonacci numbers, as most of our results are analogous to classical theorems on Fibonacci numbers. This familiarity may help the readers to grasp the motivation behind the results. Some selected articles that may be useful include \cite{Castellanos_86, Castellanos_89, Falcon_Plaza_09, Hoggatt_Bruggles_64, Ribenboim_04, Wall_60}.

\section{Generalized continuants} \label{Sec:Generalized_Continuants}

Suppose that $ (a_\nu) $ and $ (b_\nu) $ are sequences of positive integers. Consider the generalized continued fraction
\begin{align*}
	\xi_\lambda &= b_\lambda + \frac{a_{\lambda + 1}}{b_{\lambda + 1}} \nobarfrac{}{+} \frac{a_{\lambda + 2}}{b_{\lambda + 2}} \nobarfrac{}{+ \ldots} = b_\lambda + \cfrac{a_{\lambda + 1}}{b_{\lambda + 1} + \cfrac{a_{\lambda + 2}}{b_{\lambda + 2} + \ldots}}.
\end{align*}
The convergents to the continued fraction are the numbers obtained by truncating the continued fraction. For example, the $ \nu $-th convergent of $ \xi_\lambda $ is exactly  
\begin{align*}
	\xi_{\nu, \lambda} = b_\lambda + \frac{a_{\lambda + 1}}{b_{\lambda + 1}} \nobarfrac{}{+ \ldots +} \frac{a_{\lambda + \nu}}{b_{\lambda + \nu}}.
\end{align*}
Define the $ \nu $-th generalized continuants
\begin{align*}
	A_{\nu, \lambda} &= K \binom{a_{\lambda + 1}, \ldots, a_{\lambda + \nu}}{b_\lambda, b_{\lambda + 1}, \ldots, b_{\lambda + \nu}}, \\
	B_{\nu, \lambda} &= K \binom{a_{\lambda + 2}, \ldots, a_{\lambda + \nu}}{b_{\lambda + 1}, b_{\lambda + 2}, \ldots, b_{\lambda + \nu}},
\end{align*}
as the numerator and denominator, respectively, of the $ \nu $-th convergent. In other words,
\begin{align*}
	\frac{A_{\nu, \lambda}}{B_{\nu, \lambda}} = b_\lambda + \frac{a_{\lambda + 1}}{b_{\lambda + 1}} \nobarfrac{}{+ \ldots +} \frac{a_{\lambda + \nu}}{b_{\lambda + \nu}}.
\end{align*}
By the choice of sequences $ (a_\nu) $ and $ (b_\nu) $, both $ A_{\nu, \lambda} $ and $ B_{\nu, \lambda} $ are integers. We can slightly extend $ \nu $ to negative numbers by setting the initial values
\begin{align*}
	&A_{-1, \lambda} = 1, \quad A_{0, \lambda} = b_\lambda, \\
	&B_{-1, \lambda} = 0, \quad B_{0, \lambda} = 1, \quad B_{1, \lambda} = b_{\lambda + 1}.
\end{align*}
The ordinary continuants correspond to the case when $ a_\nu = 1 $ is constant. To simplify our notation, we set $ A_{\nu, 0} = A_\nu $ and $ B_{\nu, 0} = B_\nu $. Observe that $ A_{\nu, \lambda} $ and $ B_{\nu, \lambda} $ are related via the identity $ B_{\nu, \lambda} = A_{\nu - 1, \lambda + 1} $. Furthermore, there are more general identities connecting $ A_{\nu, \lambda} $ and $ B_{\nu, \lambda} $, notably an analogue of Cassini's identity for Fibonacci numbers.

\begin{theorem}[Cassini's Identity] \label{Th:Cassini_Identity} For non-negative integers $ \lambda, \nu $ and $ \mu $, we have
\begin{align*}	
	A_{\nu + \lambda + \mu - 1} B_{\nu - 1, \lambda} &= A_{\nu + \lambda - 1} B_{\nu + \mu - 1, \lambda} + (-1)^{\nu - 1} a_{\lambda + 1} \cdots a_{\lambda + \nu} A_{\lambda - 1} B_{\mu - 1, \nu + \lambda}, \\
	B_{\nu + \lambda + \mu - 1} B_{\nu - 1, \lambda} &= B_{\nu + \lambda - 1} B_{\nu + \mu - 1, \lambda} + (-1)^{\nu - 1} a_{\lambda + 1} \cdots a_{\lambda + \nu} B_{\lambda - 1} B_{\mu - 1, \nu + \lambda}.
\end{align*}
\end{theorem}

\begin{corollary}[Catalan's Identity] \label{Th:Catalan_Identity} For non-negative integers $ \lambda $ and $ \nu $, we have
\begin{align*}
	A_{\nu + \lambda} &= A_\lambda B_{\nu, \lambda} + a_{\lambda + 1} A_{\lambda - 1} B_{\nu - 1, \lambda + 1}, \\
	B_{\nu + \lambda} &= B_\lambda B_{\nu, \lambda} + a_{\lambda + 1} B_{\lambda - 1} B_{\nu - 1, \lambda + 1}.
	\end{align*}
\end{corollary}

\begin{corollary}[d'Ocagne's Identity] \label{Th:d'Ocagne_Identity} For non-negative integers $ \lambda $ and $ \nu $, we have
\begin{align*}
	A_\lambda B_{\nu - 1, \lambda - \nu} &= A_{\lambda - 1} B_{\nu, \lambda - \nu} + (-1)^{\nu - 1} a_\lambda \cdots a_{\lambda - \nu + 1} A_{\lambda - \nu - 1}, \\
	B_\lambda B_{\nu - 1, \lambda - \nu} &= B_{\lambda - 1} B_{\nu, \lambda - \nu} + (-1)^{\nu - 1} a_\lambda \cdots a_{\lambda - \nu + 1} B_{\lambda - \nu - 1}.
\end{align*}
\end{corollary}

\begin{corollary} \label{Th:Recurrence_Relation} For integer $ \nu \geq -1 $, we have
\begin{align*}
	A_{\nu + 2} &= b_{\nu + 2} A_{\nu + 1} + a_{\nu + 2} A_\nu, \\
	B_{\nu + 2} &= b_{\nu + 2} B_{\nu + 1} + a_{\nu + 2} B_\nu. 
\end{align*}
\end{corollary}

\begin{corollary} \label{Th:Index_Changing} For non-negative integers $ \lambda $ and $ \nu $, we have
\begin{align*}
	A_{\nu, \lambda} &= b_\lambda A_{\nu - 1, \lambda + 1} + a_{\lambda + 1} A_{\nu - 2, \lambda + 2}, \\
	B_{\nu, \lambda} &= b_{\lambda + 1} B_{\nu - 1, \lambda + 1} + a_{\lambda + 2} B_{\nu - 2, \lambda + 2}.
\end{align*}
\end{corollary}

From Corollary \ref{Th:Recurrence_Relation}, we can express $ A_\nu $ and $ B_\nu $ as a determinant of tridiagonal matrices.
\begin{corollary} \label{Th:Determinant_Form} For non-negative integer $ \nu $, we have
\begin{align*}
	A_\nu &= \begin{vmatrix} b_0 & -1 & 0 & \ldots & 0 & 0 \\ a_1 & b_1 & -1 & \ldots & 0 & 0 \\ 0 & a_2 & b_2 & \ldots & 0 & 0 \\ \vdots & \vdots & \vdots & \ddots & \vdots & \vdots \\ 0 & 0 & 0 & \ldots & b_{\nu - 1} & -1 \\ 0 & 0 & 0 & \ldots & a_\nu & b_\nu \end{vmatrix}, \quad B_\nu = \begin{vmatrix} b_1 & -1 & 0 & \ldots & 0 & 0 \\ a_2 & b_2 & -1 & \ldots & 0 & 0 \\ 0 & a_3 & b_3 & \ldots & 0 & 0 \\ \vdots & \vdots & \vdots & \ddots & \vdots & \vdots \\ 0 & 0 & 0 & \ldots & b_{\nu - 1} & -1 \\ 0 & 0 & 0 & \ldots & a_\nu & b_\nu \end{vmatrix}.
\end{align*}
\end{corollary}

Alternatively, we can write $ A_\nu $ and $ B_\nu $ as a product of matrices. 
\begin{corollary} For non-negative integer $ \nu $, we have
\begin{align*}
	\begin{pmatrix} A_\nu & A_{\nu - 1} \\ B_\nu & B_{\nu - 1} \end{pmatrix} &= \begin{pmatrix} 1 & 1 \\ 1 & 0 \end{pmatrix} \prod_{j = 1}^\nu \begin{pmatrix} b_j & 1 \\ a_j & 0 \end{pmatrix}.
\end{align*}
\end{corollary}

\begin{proof} The proofs of all the results above can be found in Perron's book \cite{Perron_13}, so we omit them here. It is worth noting that it is easier to prove the corollaries before proving Theorem \ref{Th:Cassini_Identity}.
\end{proof}

\begin{remark} The results presented here exhibit similarities to those of Fibonacci numbers. This connection is not surprising since choosing $ (a_\nu) = (b_\nu) = (1) $ leads to $ (B_\nu) = (F_{\nu + 1}) $, which is just the usual Fibonacci sequence.
\end{remark}

\section{Periodic continuants} \label{Sec:Periodic_Continuant}

In this section and the remainder of this paper, we consider the case when $ (a_\nu) $ and $ (b_\nu) $ are both periodic with the same period $ d $ for $ \nu \geq 1 $, i.e., $ a_{\nu + d} = a_\nu $ and $ b_{\nu + d} = b_\nu $ for all $ \nu \geq 1 $. This type of continued fraction arises frequently in various contexts. For instance, when $ N $ is a positive integer, the continued fraction expansion of $ \sqrt{N} $ is periodic. Under this assumption,
\begin{align*}
	K \binom{a_{d + 2}, \ldots, a_{d + \nu}}{b_{d + 1}, b_{d + 2}, \ldots, b_{d + \nu}} = K \binom{a_2, \ldots, a_\nu}{b_1, b_2, \ldots, b_\nu},
\end{align*}
so $ B_{\nu, d} = B_\nu $ (in fact, $ A_{\nu, \lambda} $ and $ B_{\nu, \lambda} $ are periodic with period $ d $ in the second argument). Using this fact and  Corollary \ref{Th:d'Ocagne_Identity}, we can derive an identity that will be useful later on.
\begin{corollary} \label{Th:Telescoping_Sum} For non-negative integers $ \lambda $ and $ \nu $ with $ d \mid (\lambda - \nu) $, we have
\begin{align*}
	A_{\lambda - 1} B_\nu - A_\lambda B_{\nu - 1} &= (-1)^\nu a_\lambda \cdots a_{\lambda - \nu + 1} A_{\lambda - \nu - 1}, \\
	B_{\lambda - 1} B_\nu - B_\lambda B_{\nu - 1} &= (-1)^\nu a_\lambda \cdots a_{\lambda - \nu + 1} B_{\lambda - \nu - 1}.
\end{align*}
\end{corollary}

\subsection{An equivalent recurrence relation} \label{Sec:An_Equivalent_Recurrence_Relation}

When $ (a_\nu) $ and $ (b_\nu) $ are periodic with period $ d $, we can write $ (B_\nu) $ as a second-order linear recurrence relation, 
\begin{align*}
	B_\nu = b_r B_{\nu - 1} + a_r B_{\nu - 2}, \quad \text{where } \nu \equiv r \pmod{d},
\end{align*}
with the appropriate initial values. Motivated by Carson's result \cite{Carson_07}, we seek a recurrence relation of the form
\begin{align*}
	B_{\nu + 2d} = C_d B_{\nu + d} + D_d B_\nu,
\end{align*}
for some integers $ C_d $ and $ D_d $. The next theorem provides the desired values.
\begin{theorem} \label{Th:Single_Reccurence} If $ (a_\nu) $ and $ (b_\nu) $ are periodic with period $ d $, then 
\begin{align*}
	B_{\nu + 2d} = \frac{B_{2d - 1}}{B_{d - 1}} B_{\nu + d} + (-1)^{d - 1} a_1 \cdots a_d B_\nu.
\end{align*}
\end{theorem}

\begin{proof}
Using Corollary \ref{Th:Catalan_Identity} with $ \lambda = d $ and $ 2d $, we get
\begin{align*}	
	B_{\nu + d} &= B_d B_\nu + a_1 B_{d - 1} B_{\nu - 1, 1}, \\
	B_{\nu + 2d} &= B_{2d} B_\nu + a_1 B_{2d - 1} B_{\nu - 1, 1}.
\end{align*}
By canceling out $ a_1 B_{\nu - 1, 1} $, we obtain
\begin{align*}
	B_{\nu + 2d} &= B_{2d} B_\nu + B_{2d - 1} \bigg( \frac{B_{\nu + d} - B_d B_\nu}{B_{d - 1}} \bigg). \\
	&= \frac{B_{2d - 1}}{B_{d - 1}} B_{\nu + d} + \bigg( \frac{B_{d - 1} B_{2d} - B_{2d - 1} B_d}{B_{d - 1}} \bigg) B_\nu. 
\end{align*}
Again using Corollary \ref{Th:Catalan_Identity} with $ \lambda = d $ and $ \nu = d - 1 $, we get
\begin{align*}
	B_{2d - 1} &= B_d B_{d - 1} + a_1 B_{d - 1} B_{d - 2, 1},
\end{align*}
so $ C_d = B_{2d - 1}/B_{d - 1} = B_d + a_1 B_{d - 2, 1} $ is a non-negative integer. On the other hand, using Corollary \ref{Th:Telescoping_Sum} with $ \lambda = 2d $ and $ \nu = d $, we get
\begin{align*}
	B_{2d} B_{d - 1} - B_{2d - 1} B_d &= (-1)^{d - 1} a_1 \cdots a_d B_{d - 1}.
\end{align*}
Thus, $ D_d = (-1)^{d - 1} a_1 \cdots a_d $ is an integer.
\end{proof}

One can check that the constants are indeed the same as those found by Carson \cite{Carson_07}. From Theorem \ref{Th:Single_Reccurence}, it is immediate to deduce the next corollary.
\begin{corollary} Let $ \Delta = C_d^2 + 4D_d $. For $ n \geq 1 $, we have
\begin{align*}
	B_{(n + 1)d - 1} = \frac{C_d B_{nd - 1} + \sqrt{\Delta B_{nd - 1}^2 + 4(-D_d)^n B_{d - 1}^2}}{2}.
\end{align*}
\end{corollary}

\begin{proof}
If we choose $ \lambda = \mu = d $ and $ \nu = nd $ with $ n \geq 1 $ in Theorem \ref{Th:Cassini_Identity}, then
\begin{align*}
	B_{(n + 1)d - 1}^2 - B_{nd - 1} B_{(n + 2)d - 1} = B_{(n + 1)d - 1}^2 - C_d B_{nd - 1} B_{(n + 1)d - 1} - D_d B_{nd - 1}^2 = (-D_d)^n B_{d - 1}^2.
\end{align*}
Solving this quadratic equation for $ B_{(n + 1)d - 1} $ gives the desired equation.
\end{proof}

\begin{remark} For the corresponding identities for Lucas sequences, refer to the book by Andrica and Bagdasar \cite{Andrica_Bagdasar_20}.	
\end{remark}

\subsection{Generating function and Binet's Formula} \label{Sec:Generating_Function_and_Binet's_Formula}

A benefit of having a single recurrence relation for $ (B_\nu) $ is that we can compute its generating function and Binet formula easily. We follow the proof by Panario et al.\ \cite{Panario_Sahin_Wang_13} for the next two theorems. Firstly, let 
\begin{align*}
	G(x) = \sum_{n = 0}^\infty B_n x^n,
\end{align*}
be the required generating function. Then
\begin{align*}
	(1 - C_d x^d - D_d x^{2d}) G(x) &= \sum_{n = 0}^\infty B_n x^n - C_d \sum_{n = d}^\infty B_{n - d} x^n - D_d \sum_{n = 2d}^\infty D_d B_{n - 2d} x^n \\
	&= \sum_{n = 0}^{2d - 1} B_n x^n - C_d \sum_{n = 0}^{d - 1} B_n x^{n + d}.
\end{align*}
Therefore, we arrive at our generating function.
\begin{theorem} The generating function of $ B_n $ is 
\begin{align*}
	G(x) = \frac{\sum\limits_{n = 0}^{2d - 1} B_n x^n - C_d \sum\limits_{n = 0}^{d - 1} B_n x^{n + d}}{1 - C_d x^d - D_d x^{2d}}.
\end{align*}
\end{theorem}

Using this generating function, we can then calculate the Binet formula for $ B_\nu $. Since the continued fraction has period $ d $, it is perhaps not surprising that the Binet formula depends on the residue modulo $ d $. Let $ \Delta = C_d^2 + 4D_d $ be the discriminant of the equation $ D_d z^2 + C_d z - 1 = 0 $. We assume throughout this paper that $ \Delta \neq 0 $. The next theorem gives the Binet formula for $ (B_\nu) $.
\begin{theorem}[Binet's Formula] \label{Th:Binet_Formula} Suppose that $ \Delta \neq 0 $. Let $ \alpha $ and $ \beta $ be the positive and negative roots of $ D_d z^2 + C_d z - 1 = 0 $, respectively. For integers $ n $ and $ r $ such that $ n \geq 0 $ and $ r \geq -1 $, we have
\begin{align*}
	B_{nd + r} = (-D_d)^{n - 1} \bigg( \frac{\alpha^n - \beta^n}{\alpha - \beta} B_{d + r} - \frac{\alpha^{n - 1} - \beta^{n - 1}}{\alpha - \beta} B_r \bigg).
\end{align*}
\end{theorem}

\begin{proof}
Consider the generating function of the subsequence $ (B_{nd + r}) $ given by
\begin{align*}
	G_r (x) = \frac{B_r x^r + B_{d + r} x^{d + r} - C_d B_r x^{d + r}}{1 - C_d x^d - D_d x^{2d}}.
\end{align*}
Let $ \alpha $ and $ \beta $ be the positive and negative roots of $ D_d z^2 + C_d z - 1 = 0 $, respectively, or equivalently,
\begin{align*}
	\alpha = \frac{-C_d + \sqrt{\Delta}}{2D_d}, \quad \beta = \frac{-C_d - \sqrt{\Delta}}{2D_d}.
\end{align*}
Note that, $ \alpha + \beta = -C_d/D_d $ and $ \alpha \beta = -1/D_d $. By partial fraction decomposition, we have
\begin{align*}
	\frac{1}{1 - C_d x^d - D_d x^{2d}} = \frac{1}{\alpha - \beta} \bigg( \frac{\alpha}{1 + D_d \alpha x^d} - \frac{\beta}{1 + D_d \beta x^d} \bigg).	
\end{align*}
Substituting this into $ G_r (x) $, we obtain
\begin{align*}
	G_r (x) &= \frac{B_r x^r + B_{d + r} x^{d + r} - C_d B_r x^{d + r}}{\alpha - \beta} \sum_{n = 0}^\infty (-1)^n D_d^n (\alpha^{n + 1} - \beta^{n + 1}) x^{nd} \\
	&= \sum_{n = 0}^\infty (-1)^n \frac{D_d^n (\alpha^{n + 1} - \beta^{n + 1}) B_r}{\alpha - \beta} x^{nd + r} + \sum_{n = 0}^\infty (-1)^n \frac{D_d^n (\alpha^{n + 1} - \beta^{n + 1}) (B_{d + r} - C_d B_r)}{\alpha - \beta} x^{(n + 1)d + r}. 
\end{align*}
By replacing $ n $ with $ n - 1 $ in the second sum, it becomes
\begin{align*}
	\sum_{n = 0}^\infty (-1)^{n - 1} \frac{D_d^{n - 1} (\alpha^n - \beta^n) (B_{d + r} - C_d B_r)}{\alpha - \beta} x^{nd + r}.
\end{align*}
This combined with the first sum gives the expression
\begin{align*}
	G_r (x) = \sum_{n = 0}^\infty (-1)^n \bigg( \frac{D_d^n (\alpha^{n + 1} - \beta^{n + 1}) B_r - D_d^{n - 1} (\alpha^n - \beta^n) (B_{d + r} - C_d B_r)}{\alpha - \beta} \bigg) x^{nd + r}.
\end{align*}
The terms containing $ B_r $ simplifies to
\begin{align*}
	&D_d^n (\alpha^{n + 1} - \beta^{n + 1}) + C_d D_d^{n - 1} (\alpha^n - \beta^n) = D_d^{n - 1} (\alpha^{n - 1} - \beta^{n - 1}).
\end{align*}
Therefore, we arrive at
\begin{align*}
	G_r (x) = \sum_{n = 0}^\infty (-D_d)^{n - 1} \bigg( \frac{\alpha^n - \beta^n}{\alpha - \beta} B_{d + r} - \frac{\alpha^{n - 1} - \beta^{n - 1}}{\alpha - \beta} B_r \bigg) x^{nd + r}.
\end{align*}
The theorem holds by comparing the coefficient of $ x^{nd + r} $.
\end{proof}

In particular, when $ r = -1 $, one observes that $ B_{nd - 1}/B_{d - 1} $ and $ B_{2nd - 1}/B_{nd - 1} $ behave like the Lucas sequence of the first kind and second kind, respectively. Hence, new identities can be discovered by utilizing known identities involving Fibonacci and Lucas numbers. For example, the identity $ L_n^2 - 5F_n^2 = 4(-1)^n $ yields the following identity, which can be verified directly: 
\begin{align*}
	\frac{B_{2nd - 1}^2}{B_{nd - 1}^2} - \frac{\Delta B_{nd - 1}^2}{B_{d - 1}^2} -  &= 4(-D_d)^n.
\end{align*}

\begin{remark} We also have another related identity
\begin{align*}
	\frac{B_{2nd - 1}^2}{B_{nd - 1}^2} + \frac{\Delta B_{nd - 1}^2}{B_{d - 1}^2} &= \frac{2B_{4nd - 1}^2}{B_{2nd - 1}^2}. 
\end{align*}
\end{remark}

\subsubsection{Example: bi-periodic Fibonacci numbers}

When $ d = 2 $, the recurrence relation for $ (B_\nu) $ can be written as 
\begin{align*}
	B_\nu = \begin{cases} b_1 B_{\nu - 1} + a_1 B_{\nu - 2}, \quad \text{ if } \nu \text{ is odd}; \\ b_2 B_{\nu - 1} + a_2 B_{\nu - 2}, \quad \text{ if } \nu \text{ is even}, \end{cases}
\end{align*}
with initial values $ B_{-1} = 0 $ and $ B_0 = 1 $. The first few terms of $ B_\nu $ are
\begin{align*}
	&B_1 = b_1, \quad B_2 = b_1 b_2 + a_2, \quad B_3 = b_1^2 b_2 + a_1 b_1 + a_2 b_1, \\
	&B_4 = b_1^2 b_2^2 + a_2^2 + 2a_2 b_1 b_2 + a_1 b_1 b_2. 
\end{align*}
We calculate
\begin{align*}
	C_2 = \frac{B_3}{B_1} &= b_1 b_2 + a_1 + a_2, \quad D_2 = -a_1 a_2.
\end{align*}
Hence, $ (B_\nu) $ can be defined using a single recurrence relation 
\begin{align*}
	B_\nu = (b_1 b_2 + a_1 + a_2) B_{\nu - 2} - a_1 a_2 B_{\nu - 4},
\end{align*}
with initial values as stated above. Its generating function is 
\begin{align*}
	G(x) = \frac{1 + b_1 x - a_2 x^2}{1 - (b_1 b_2 + a_1 + a_2) x^2 + a_1 a_2 x^4},
\end{align*}
and its Binet formula is 
\begin{align*}
	B_{2n + r} = (a_1 a_2)^{n - 1} \bigg( \frac{\alpha^n - \beta^n}{\alpha - \beta} B_{r + 2} - \frac{\alpha^{n - 1} - \beta^{n - 1}}{\alpha - \beta} B_r \bigg).
\end{align*}
where $ \alpha $ and $ \beta $ are the positive and negative roots of $ a_1 a_2 z^2 - (b_1 b_2 + a_1 + a_2)z + 1 = 0 $, respectively.

Consider the convergents to the continued fraction of $ \sqrt{8} = [2, \overline{1, 4}] $. The denominators of these convergents are given by the sequence $ 1, 1, 5, 6, 29, 35, 169, 204, 985, \ldots $ (see \seqnum{A041011} on OEIS). This sequence is generated by the following second-order linear recurrence relation:
\begin{align*}
	B_\nu = 6B_{\nu - 2} - B_{\nu - 4}.
\end{align*}
Its generating function is 
\begin{align*}
	G(x) = \frac{1 + x - x^2}{1 - 6x^2 + x^4}.
\end{align*}
The roots of $ 1 - 6x + x^2 = 0 $ are $ 3 \pm 2\sqrt{2} $. Thus, the Binet formula is given by
\begin{align*}
	B_{2n + r} = \frac{(3 + 2 \sqrt{2})^n - (3 - 2 \sqrt{2})^n}{4 \sqrt{2}} B_{r + 2} - \frac{(3 + 2 \sqrt{2})^{n - 1} - (3 - 2 \sqrt{2})^{n - 1}}{4 \sqrt{2}} B_r.
\end{align*}
We will return to this example a few more times throughout this paper.

\subsection{Limits}

We use Theorem \ref{Th:Binet_Formula} to compute the limit of consecutive term, which we need for later. Observe that, $ |\alpha| < |\beta| $. Hence, we deduce the next two theorems.
\begin{theorem} \label{Th:Limit_of_Consecutive_Terms} Let $ \alpha $ and $ \beta $ be as defined in Theorem \ref{Th:Binet_Formula}. For integer $ r \geq 0 $, we have
\begin{align*}
	\lim_{n \rightarrow \infty} \frac{B_{nd + r}}{B_{nd + r - 1}} = \frac{\beta B_{d + r} - B_r}{\beta B_{d + r - 1} - B_{r - 1}}.
\end{align*}
\end{theorem}

\begin{proof} Take $ n \longrightarrow \infty $ in the expression below
\begin{align*}
	\frac{B_{nd + r}}{B_{nd + r - 1}} = \frac{(\alpha (\alpha/\beta)^{n - 1} - \beta) B_{d + r} - ((\alpha/\beta)^{n - 1} - 1) B_r}{(\alpha (\alpha/\beta)^{n - 1} - \beta) B_{d + r - 1} - ((\alpha/\beta)^{n - 1} - 1) B_{r - 1}},
\end{align*}
and use the fact that $ (\alpha/\beta)^n \longrightarrow 0 $.
\end{proof}

\begin{theorem} \label{Th:Limit_of_Consecutive_Periods} Let $ \alpha $ and $ \beta $ be as defined in Theorem \ref{Th:Binet_Formula}. For integer $ r \geq -1 $, we have
\begin{align*}
	\lim_{n \rightarrow \infty} \frac{B_{(n + 1)d + r}}{B_{nd + r}} = -D_d \beta.
\end{align*}
\end{theorem}

\begin{proof} Take $ n \longrightarrow \infty $ in the expression below
\begin{align*}
	\frac{B_{(n + 1)d + r}}{B_{nd + r}} = \frac{-D_d (\alpha^2 (\alpha/\beta)^{n - 1} - \beta^2) B_{d + r} - (\alpha (\alpha/\beta)^{n - 1} - \beta) B_r}{(\alpha (\alpha/\beta)^{n - 1} - \beta) B_{d + r} - ((\alpha/\beta)^{n - 1} - 1) B_r},
\end{align*}
and use the fact that $ (\alpha/\beta)^n \longrightarrow 0 $.
\end{proof}

\subsection{Extending to negative indices} \label{Sec:Extending_to_negative_indices}

Replacing $ n $ with $ -n $ in Binet's formula (Theorem \ref{Th:Binet_Formula}) yields the next theorem.
\begin{theorem} For integers $ n $ and $ r $ such that $ n \geq 0 $ and $ r \geq -1 $, we have
\begin{align*}
	B_{-nd + r} = -(-D_d)^{n - 1} \bigg( \frac{\alpha^n - \beta^n}{(-D_d)^n (\alpha - \beta)} B_{d + r} - \frac{\alpha^{n + 1} - \beta^{n + 1}}{(-D_d)^{n - 1} (\alpha - \beta)} B_r \bigg).
\end{align*}
In particular, when $ r = - 1 $, $ (-D_d)^n B_{-nd - 1} = -B_{nd - 1} $.
\end{theorem}

\subsection{Some curious telescoping sums} \label{Sec:Some_Curious_Telescoping_Sums}

Many interesting sums involving Fibonacci numbers arise from telescoping sums. We demonstrate here some of the analogous sums based on continuants.

\subsubsection{Telescoping sum 1}

When $ d = 2 $, we have encountered this sequence before. It has the form
\begin{align*}
	B_\nu = \begin{cases} b_1 B_{\nu - 1} + a_1 B_{\nu - 2}, \quad \text{ if } \nu \text{ is odd}; \\ b_2 B_{\nu - 1} + a_2 B_{\nu - 2}, \quad \text{ if } \nu \text{ is even}, \end{cases}
\end{align*}
with initial values $ B_{-1} = 0 $ and $ B_0 = 1 $. The next theorem is a continuant version of the Millin series \cite{Millin_74} for Fibonacci numbers.

\begin{theorem} If $ d = 2 $, then
\begin{align*}
	\sum_{n = 1}^\infty \frac{(a_1 a_2)^{2^{n - 1}}}{B_{2^{n + 1} - 1}} &= \frac{1}{b_1 \beta}.
\end{align*}
\end{theorem}

\begin{proof} By choosing $ \lambda = 2^{n + 1} $ and $ \nu = 2^n $ with $ n \geq 1 $ in Corollary \ref{Th:Telescoping_Sum}, we get
\begin{align*}
	\frac{B_{2^n}}{B_{2^n - 1}} - \frac{B_{2^{n + 1}}}{B_{2^{n + 1} - 1}} &= \frac{(a_1 a_2)^{2^{n - 1}}}{B_{2^{n + 1} - 1}}.
\end{align*}
The LHS is telescoping, hence summing from $ 1 $ to $ N $, we have
\begin{align*}
	\frac{B_2}{B_1} - \frac{B_{2^{N + 1}}}{B_{2^{N + 1} - 1}} &= \sum_{n = 1}^N \frac{(a_1 a_2)^{2^{n - 1}}}{B_{2^{n + 1} - 1}}.
\end{align*}
When $ N \longrightarrow \infty $, the second term on the LHS approaches a finite limit
\begin{align*}
	\frac{\beta B_2 - 1}{\beta B_1},
\end{align*}
by Theorem \ref{Th:Limit_of_Consecutive_Terms}. Therefore,
\begin{align*}
	\sum_{n = 1}^\infty \frac{(a_1 a_2)^{2^{n - 1}}}{B_{2^{n + 1} - 1}} &= \frac{B_2}{B_1} - \frac{\beta B_2 - 1}{\beta B_1} = \frac{1}{b_1 \beta}.
\end{align*}
\end{proof}

In the case of $ \sqrt{8} $, this sum becomes
\begin{align*}
	\sum_{n = 1}^\infty \frac{1}{B_{2^{n + 1} - 1}} = \frac{1}{6} + \frac{1}{204} + \frac{1}{235416} + \ldots = \frac{1}{3 + 2\sqrt{2}} = 3 - 2\sqrt{2}.
\end{align*}

\subsubsection{Telescoping sum 2}

For general $ d $, $ (B_\nu) $ has the form
\begin{align*}
	B_\nu = b_r B_{\nu - 1} + a_r B_{\nu - 2}, \quad \text{where } \nu \equiv r \pmod{d},
\end{align*}
for the appropriate initial values. 

\begin{theorem} If the period is $ d $, then
\begin{align*}
	\sum_{n = 1}^\infty \frac{(-D_d)^{n - 1}}{B_{nd - 1} B_{(n + 1)d - 1}} = \frac{\alpha}{B_{d - 1}^2}.
\end{align*}
\end{theorem}

\begin{proof} By choosing $ \lambda = \mu = d $ and $ \nu = nd $ with $ n \geq 1 $, Theorem \ref{Th:Cassini_Identity} becomes
\begin{align*}
	\frac{B_{(n + 1)d - 1}}{B_{nd - 1}} - \frac{B_{(n + 2)d - 1}}{B_{(n + 1)d - 1}} &= \frac{(-1)^{nd} (a_1 \cdots a_d)^n B_{d - 1}^2}{B_{nd - 1} B_{(n + 1)d - 1}}.
\end{align*}
The LHS is telescoping, so summing from $ 1 $ to $ N $, we get
\begin{align*}
	\frac{B_{2d - 1}}{B_{d - 1}} - \frac{B_{(N + 2)d - 1}}{B_{(N + 1)d - 1}}  &= \sum_{n = 1}^N \frac{(-1)^{nd} (a_1 \cdots a_d)^n B_{d - 1}^2}{B_{nd - 1} B_{(n + 1)d - 1}}.
\end{align*}
When $ N \longrightarrow \infty $, the second term on the LHS approaches $ \beta $ by Theorem \ref{Th:Limit_of_Consecutive_Periods}. Therefore, 
\begin{align*}
	\sum_{n = 1}^\infty \frac{(-1)^{nd} (a_1 \cdots a_d)^n B_{d - 1}^2}{B_{nd - 1} B_{(n + 1)d - 1}} = C_d + D_d \beta = -D_d \alpha.
\end{align*}
This proves the theorem after simplifications.
\end{proof}

It is a known fact that, for a positive integer $ N $, the continued fraction of $ \sqrt{N} $ is periodic (refer to Khrushchev's book \cite{Khrushchev_08} for a proof), say with period $ d $. When $ (B_\nu) $ represents the denominators of the convergents to $ \sqrt{N} $, we have $ C_d = B_{2d - 1}/B_{d - 1} = 2A_{d - 1} $ and $ D_d = (-1)^{d - 1} $. This implies that $ \alpha = (-1)^d A_{d - 1} - (-1)^d \sqrt{A_{d - 1}^2 - (-1)^d} $. Furthermore, there is an interesting interpretation of the number $ B_{nd - 1} $ as a solution to the Pell's equation $ x^2 - Ny^2 = 1 $, as stated in the next two theorems.
\begin{theorem} The fundamental (minimal) solution to $ x^2 - Ny^2 = 1 $ is 
\begin{align*}
	(x_1, y_1) = \begin{cases} (A_{d - 1}, B_{d - 1}), &\text{ if } d \text{ is even}; \\ (A_{2d - 1}, B_{2d - 1}), &\text{ if } d \text{ is odd}. \end{cases}
\end{align*}
\end{theorem}

\begin{theorem} All the solutions to $ x^2 - Ny^2 = 1 $ are given by
\begin{align*}
	(x_n, y_n) = \begin{cases} (A_{nd - 1}, B_{nd - 1}), &\text{ if } d \text{ is even}; \\ (A_{2nd - 1}, B_{2nd - 1}), &\text{ if } d \text{ is odd}, \end{cases}
\end{align*}
where $ n $ is a positive integer.
\end{theorem}

\begin{proof} See Khrushchev's book \cite{Khrushchev_08} for proofs.
\end{proof}

Hence, our telescoping sum can be reinterpreted as the next corollary.
\begin{corollary} If $ d $ is even and $ (x_n, y_n) $ is the $ n $-th solution to $ x^2 - Ny^2 = 1 $, then
\begin{align*}
	\sum_{n = 1}^\infty \frac{1}{y_n y_{n + 1}} = \frac{x_1 - \sqrt{x_1^2 - 1}}{y_1^2}.
\end{align*}	
\end{corollary}

In the case when $ N = 8 $, the sum becomes 
\begin{align*}
\sum_{n = 1}^\infty \frac{1}{y_n y_{n + 1}} = \frac{1}{6} + \frac{1}{210} + \frac{1}{7140} + \ldots = 3 - \sqrt{8}.
\end{align*}

\subsubsection{Telescoping sum 3}

We assume for this part	that $ (B_\nu) $ is the sequence of denominators of the convergents to $ \sqrt{N} $. In this case, $ a_n = 1 $ is constant and $ D_d = (-1)^{d - 1} $. The next corollary is similar to the previous one.
\begin{corollary} If $ d $ is even and $ (x_n, y_n) $ is the $ n $-th solution to $ x^2 - Ny^2 = 1 $, then
\begin{align*}
	\sum_{n = 1}^\infty \frac{1}{x_n x_{n + 1}} = \frac{x_1 - \sqrt{x_1^2 - 1}}{x_1 \sqrt{x_1^2 - 1}}.
\end{align*}
\end{corollary}

\begin{proof} If we choose  $ \lambda = \mu = d $ and $ \nu = (n - 1)d $ with $ n \geq 1 $, then Theorem \ref{Th:Cassini_Identity} becomes
\begin{align*}	
	\frac{B_{nd - 1}}{A_{(n + 1)d - 1}} - \frac{B_{(n - 1)d - 1}}{A_{nd - 1}} = \frac{(-1)^{(n - 1)d} (a_1 \cdots a_d)^{n - 1} A_{d - 1} B_{d - 1}}{A_{nd - 1} A_{(n + 1)d - 1}}.
\end{align*}
The LHS is telescoping, so summing from $ 1 $ to $ N $, we get
\begin{align*}
	\frac{B_{Nd - 1}}{A_{(N + 1)d - 1}} - \frac{B_{-1}}{A_{d - 1}} = \frac{B_{Nd - 1}}{A_{(N + 1)d - 1}} = \sum_{n = 1}^N \frac{(-1)^{(n - 1)d} (a_1 \cdots a_d)^{n - 1} A_{d - 1} B_{d - 1}}{A_{nd - 1} A_{(n + 1)d - 1}}.
\end{align*}
When $ N \longrightarrow \infty $, using the same procedure as in Theorem  \ref{Th:Limit_of_Consecutive_Periods}, one checks that
\begin{align*}
	\frac{B_{Nd - 1}}{A_{(N + 1)d - 1}} \rightarrow -\frac{2B_{d - 1}}{D_d^2 \beta (\alpha - \beta)}.
\end{align*}
Therefore,
\begin{align*}
	\sum_{n = 1}^\infty \frac{(-1)^{(n - 1)d} (a_1 \cdots a_d)^{n - 1} A_{d - 1} B_{d - 1}}{A_{nd - 1} A_{(n + 1)d - 1}} = -\frac{2B_{d - 1}}{D_d^2 \beta (\alpha - \beta)},
\end{align*}
which can be simplified to 
\begin{align*}
	\sum_{n = 1}^\infty \frac{(-D_d)^n}{A_{nd - 1} A_{(n + 1)d - 1}} &= -\frac{2\alpha}{A_{d - 1} (\alpha - \beta)}.
\end{align*}
We get our conclusion by noting that $ D_d = -1 $ and $ \alpha - \beta = -2\sqrt{A_{d - 1}^2 - 1} $.
\end{proof}

In the case when $ N = 8 $, we have the sum
\begin{align*}
	\sum_{n = 1}^\infty \frac{1}{x_n x_{n + 1}} = \frac{1}{51} + \frac{1}{1683} + \frac{1}{57123} + \ldots = \frac{3\sqrt{2} - 4}{12}.
\end{align*}

\subsubsection{Telescoping sum 4}

Let $ (x_n) $ and $ (y_n) $ be any sequences . Consider the identity, which can be verified directly,
\begin{align*}
	\frac{4(x_{n + 1} y_n - x_n y_{n + 1})(x_n x_{n + 1} - y_n y_{n + 1})}{(x_n - y_n)^2 (x_{n + 1} - y_{n + 1})^2} = \frac{(x_n + y_n)^2}{(x_n - y_n)^2} - \frac{(x_{n + 1} + y_{n + 1})^2}{(x_{n + 1} - y_{n + 1})^2}.
\end{align*}
We use this identity to prove the next theorem.
\begin{theorem}  If the period is $ d $, then
\begin{align*}
	\sum_{n = 1}^\infty \frac{(-D_d)^{n - 1} B_{(2n + 1)d - 1}}{B_{nd - 1}^2 B_{(n + 1)d - 1}^2} = \frac{1}{B_{d - 1}^3}.
\end{align*}
\end{theorem}

\begin{proof}
Choose $ x_n = \alpha^n $ and $ y_n = \beta^n $ in the identity above to get
\begin{align*}
	\frac{4(\alpha - \beta)(\alpha^{2n + 1} - \beta^{2n + 1})}{(-D_d)^n (\alpha^n - \beta^n)^2 (\alpha^{n + 1} - \beta^{n + 1})^2} = \frac{(\alpha^n + \beta^n)^2}{(\alpha^n - \beta^n)^2} - \frac{(\alpha^{n + 1} + \beta^{n + 1})^2}{(\alpha^{n + 1} - \beta^{n + 1})^2}.
\end{align*}
The RHS is telescoping, so summing from $ n = 1 $ to $ N $, we obtain
\begin{align*}
	\sum_{n = 1}^N \frac{4(\alpha - \beta)(\alpha^{2n + 1} - \beta^{2n + 1})}{(-D_d)^n (\alpha^n - \beta^n)^2 (\alpha^{n + 1} - \beta^{n + 1})^2} = \frac{(\alpha + \beta)^2}{(\alpha - \beta)^2} - \frac{(\alpha^{N + 1} + \beta^{N + 1})^2}{(\alpha^{N + 1} - \beta^{N + 1})^2}.
\end{align*}
Taking $ N \longrightarrow \infty $, the second term on RHS approaches $ 1 $. Therefore,
\begin{align*}
	\sum_{n = 1}^\infty \frac{4(\alpha - \beta)(\alpha^{2n + 1} - \beta^{2n + 1})}{(-D_d)^n (\alpha^n - \beta^n)^2 (\alpha^{n + 1} - \beta^{n + 1})^2} = \frac{(\alpha + \beta)^2}{(\alpha - \beta)^2} - 1.
\end{align*}
Rewriting this equation in terms of $ B_\nu $, we obtain
\begin{align*}
	\sum_{n = 1}^N \frac{4(-D_d)^{n - 2} B_{(2n + 1)d - 1} B_{d - 1}^3}{B_{nd - 1}^2 B_{(n + 1)d - 1}^2} = \frac{C_d^2}{D_d^2} - \frac{\Delta}{D_d^2}.
\end{align*}
We conclude our theorem upon simplification of this last equation. 
\end{proof}

Suppose that $ (B_\nu) $ are the denominators of the convergents to $ \sqrt{N} $ for some positive integer $ N $. The next corollary is immediate.
\begin{corollary} If $ d $ is even and $ (x_n, y_n) $ are the $ n $-th solution to $ x^2 - Ny^2 = 1 $, then
\begin{align*}
	\sum_{n = 1}^\infty \frac{y_{2n + 1}}{y_n^2 y_{n + 1}^2} = \frac{1}{y_1^3}.
\end{align*}
\end{corollary}
When $ N = 8 $, we have 
\begin{align*}
	\sum_{n = 1}^\infty \frac{y_{2n + 1}}{y_n^2 y_{n + 1}^2} = \frac{35}{1^2 \cdot 6^2} + \frac{1189}{6^2 \cdot 35^2} + \frac{40391}{35^2 \cdot 204^2} + \ldots = 1.
\end{align*}

\subsubsection{Telescoping sum 5}
Suppose that $ D_d = 1 $, so
\begin{align*}
B_{(n + 1)d - 1} - B_{(n - 1)d - 1} = C_d B_{nd - 1}.
\end{align*}
Motivated by the identity $ F_n^2 - F_{n + 1} F_{n - 1} = (-1)^{n - 1} $ for Fibonacci numbers, we deduce
\begin{align*}
\frac{B_{nd - 1}^2}{B_{d - 1}^2} - \frac{B_{(n + 1)d - 1}}{B_{d - 1}} \cdot \frac{B_{(n - 1)d - 1}}{B_{d - 1}} &= (-1)^{n - 1},
\end{align*}
where $ n \geq 1 $. The application of this identity yields the next result.
\begin{theorem} \label{Th:Telescoping_arctan} If $ D_d = 1 $, then
\begin{align*}
	\sum_{n = 1}^\infty \tan^{-1} \frac{B_{2d - 1}}{B_{(2n + 1)d - 1}} = \tan^{-1} \frac{B_{d - 1}}{B_{2d - 1}}.
\end{align*}
\end{theorem}

\begin{proof} Using the addition formula for $ \tan^{-1} $ (see \S \ref{Sec:Series_Involving_arctan}),
\begin{align*}
	\tan^{-1} \frac{B_{d - 1}}{B_{2nd - 1}} - \tan^{-1} \frac{B_{d - 1}}{B_{2(n + 1)d - 1}} &= \tan^{-1} \frac{B_{d - 1} (B_{2(n + 1)d - 1} - B_{2nd - 1})}{B_{d - 1}^2 + B_{2(n + 1)d - 1} B_{2nd - 1}} \\
	&= \tan^{-1} \frac{B_{2d - 1}}{B_{(2n + 1)d - 1}}.
\end{align*}
The LHS is telescoping, so summing from $ 1 $ to $ N $, we obtain
\begin{align*}
	\tan^{-1} \frac{B_{d - 1}}{B_{2d - 1}} - \tan^{-1} \frac{B_{d - 1}}{B_{2(N + 1)d - 1}} = \sum_{n = 1}^N \tan^{-1} \frac{B_{2d - 1}}{B_{(2n + 1)d - 1}}.
\end{align*}
Taking $ N \longrightarrow \infty $, the second term on the LHS vanishes, so
\begin{align*}
	\sum_{n = 1}^\infty \tan^{-1} \frac{B_{2d - 1}}{B_{(2n + 1)d - 1}} = \tan^{-1} \frac{B_{d - 1}}{B_{2d - 1}}.
\end{align*}
\end{proof}

\begin{remark} This theorem generalizes an earlier theorem by Hoggatt and Bruggles \cite[Theorem 4]{Hoggatt_Bruggles_64}.
\end{remark}

Using a similar construction, we can get a corresponding sum for $ \tanh^{-1} $.
\begin{theorem} \label{Th:Telescoping_artanh} If $ D_d = 1 $, then
\begin{align*}
	\sum_{n = 2}^\infty \tanh^{-1} \frac{B_{2d - 1}}{B_{2nd - 1}} = \frac{1}{2} \ln \frac{B_{3d - 1} + B_{d - 1}}{B_{3d - 1} - B_{d - 1}}.
\end{align*}
\end{theorem}

\begin{proof} Using the addition formula for $ \tanh^{-1} $ (see \S \ref{Sec:Series_Involving_artanh}),
\begin{align*}
	\tanh^{-1} \frac{B_{d - 1}}{B_{(2n - 1)d - 1}} - \tanh^{-1} \frac{B_{d - 1}}{B_{(2n + 1)d - 1}} &= \tanh^{-1} \frac{B_{d - 1} (B_{(2n + 1)d - 1} - B_{(2n - 1)d - 1})}{B_{(2n + 1)d - 1} B_{(2n - 1)d - 1} - B_{d - 1}^2} \\
	&= \tanh^{-1} \frac{B_{2d - 1}}{B_{2nd - 1}}.
\end{align*}
The LHS is telescoping, so summing from $ 2 $ to $ N $ ($ \tanh^{-1} 1 $ is undefined), we obtain
\begin{align*}
	\tanh^{-1} \frac{B_{d - 1}}{B_{3d - 1}} - \tanh^{-1} \frac{B_{d - 1}}{B_{(2N + 1)d - 1}} = \sum_{n = 2}^N \tanh^{-1} \frac{B_{2d - 1}}{B_{2nd - 1}}.
\end{align*}
Taking $ N \longrightarrow \infty $, the second term on the LHS vanishes, so
\begin{align*}
	\sum_{n = 2}^\infty \tanh^{-1} \frac{B_{2d - 1}}{B_{2nd - 1}} = \tanh^{-1} \frac{B_{d - 1}}{B_{3d - 1}} = \frac{1}{2} \ln \frac{B_{3d - 1} + B_{d - 1}}{B_{3d - 1} - B_{d - 1}},
\end{align*}
where we used the fact that 
\begin{align*}
	\tanh^{-1} x = \frac{1}{2} \ln \frac{1 + x}{1 - x}.
\end{align*}
\end{proof}

Let us apply these theorems to the continued fraction of $ \sqrt{2} = [1, \overline{2}] $. The denominators, $ (B_\nu) $, of the convergents to $ \sqrt{2} $ are $ 1, 2, 5, 12, 29, 70, 169, 408, 985, \ldots $. Each $ B_\nu $ corresponds to the $ (\nu + 1) $-th term of the Pell number (see \seqnum{A000129} on OEIS). The theorems above imply that
\begin{align*}
	\sum_{n = 1}^\infty \tan^{-1} \frac{2}{B_{2n}} = \tan^{-1} \frac{2}{5} + \tan^{-1} \frac{2}{29} + \tan^{-1} \frac{2}{169} + \ldots = \tan^{-1} \frac{1}{2},
\end{align*}
and
\begin{align*}
	\sum_{n = 2}^\infty \tanh^{-1} \frac{2}{B_{2n - 1}} = \tanh^{-1} \frac{1}{6} + \tanh^{-1} \frac{1}{35} + \tanh^{-1} \frac{1}{204} + \ldots = \tanh^{-1} \frac{1}{5} = \frac{1}{2} \ln \frac{3}{2}.
\end{align*}

\subsection{More series involving \texorpdfstring{$ \tan^{-1} $}{arctan} and \texorpdfstring{$ \tanh^{-1} $}{artanh}} \label{Sec:More_Series_Involving_arctan_and_artanh}

Besides the telescoping sums in the last section, another set of series can be generated using $ \tan^{-1} $ and $ \tanh^{-1} $. The results in this section are inspired by the papers of Castellanos \cite{Castellanos_86, Castellanos_89}. However, Castellanos' original method using Chebyshev polynomials is cumbersome. Therefore, we provide a direct method that is applicable to his results, while generating new ones.

\subsubsection{Series involving \texorpdfstring{$ \tan^{-1} $}{arctan}} \label{Sec:Series_Involving_arctan}

 Let $ \zeta $ be a number to be determined later. Using the addition formula for $ \tan^{-1} $,
\begin{align*}
	\tan^{-1} (x - y) = \frac{\tan^{-1} x - \tan^{-1} y}{1 + \tan^{-1} x \tan^{-1} y},
\end{align*}
and the Gregory series,
\begin{align*}
	\tan^{-1} z = \sum_{k = 0}^\infty \frac{(-1)^k z^{2k + 1}}{2k + 1},
\end{align*}
we get
\begin{align*}
	\tan^{-1} \frac{(\alpha - \beta) \zeta}{\alpha \beta + \zeta^2} = \tan^{-1} \frac{\zeta}{\beta} - \tan^{-1} \frac{\zeta}{\alpha} = \sum_{k = 0}^\infty \frac{(-1)^k (\alpha^{2k + 1} - \beta^{2k + 1})}{(2k + 1) (\alpha \beta)^{2k + 1}} \zeta^{2k + 1}.
\end{align*}
Since the series expansion for $ \tan^{-1} z $ is only valid for $ |z| < 1 $, we require that $ |\zeta| < \min (|\alpha|, |\beta|) = |\alpha| $. The next two theorems are direct consequences of the last equation.
\begin{theorem} Let $ \zeta $ be a root of $ D_d z^2 - \sqrt{3\Delta} z - 1 = 0 $ that satisfies $ |\zeta| < \alpha $. Then
\begin{align*}
	\frac{\pi B_{d - 1}}{6 \sqrt{\Delta}} = \sum_{k = 0}^\infty \frac{(-1)^{k + 1} B_{(2k + 1)d - 1}}{2k + 1} \zeta^{2k + 1}.
\end{align*}
\end{theorem}

\begin{proof} We choose $ \zeta $ such that 
\begin{align*}
	\frac{(\alpha - \beta) \zeta}{\alpha \beta + \zeta^2} = \frac{1}{\sqrt{3}}.
\end{align*}
We obtain our theorem upon simplification and using the fact $ \tan^{-1} (1/\sqrt{3}) = \pi/6 $.
\end{proof}

\begin{theorem} Let $ \zeta $ be a root of $ D_d z^2 - (\sqrt{2} - 1) \sqrt{\Delta} z - 1 = 0 $ that satisfies $ |\zeta| < \alpha $. Then
\begin{align*}
	\frac{\pi B_{d - 1}}{8 \sqrt{\Delta}} = \sum_{k = 0}^\infty \frac{(-1)^{k + 1} B_{(2k + 1)d - 1}}{2k + 1} \zeta^{2k + 1}.
\end{align*}
\end{theorem}

\begin{proof} We choose $ \zeta $ such that 
\begin{align*}
	\frac{(\alpha - \beta) \zeta}{\alpha \beta + \zeta^2} = \sqrt{2} - 1.
\end{align*}
We obtain our theorem upon simplification and using the fact $ \tan^{-1} (\sqrt{2} - 1) = \pi/8 $.
\end{proof}

\subsubsection{Series involving \texorpdfstring{$ \tanh^{-1} $}{artanh}} \label{Sec:Series_Involving_artanh}

We also have similar addition formula for $ \tanh^{-1} $,
\begin{align*}
	\tanh^{-1} (x - y) = \frac{\tanh^{-1} x - \tanh^{-1} y}{1 - \tanh^{-1} x \tanh^{-1} y},
\end{align*}
and Taylor series expansion,
\begin{align*}
	\tanh^{-1} z = \sum_{k = 0}^\infty \frac{z^{2k + 1}}{2k + 1}.
\end{align*}
So one expects similar series using $ \tanh^{-1} $. Again, take $ \zeta $ as a number to be determined later, then 
\begin{align*}
\tanh^{-1} \frac{(\alpha - \beta) \zeta}{\alpha \beta - \zeta^2} = \tanh^{-1} \frac{\zeta}{\beta} - \tanh^{-1} \frac{\zeta}{\alpha} = \sum_{k = 0}^\infty \frac{\alpha^{2k + 1} - \beta^{2k + 1}}{(2k + 1) (\alpha \beta)^{2k + 1}} \zeta^{2k + 1}.
\end{align*}
The series expansion for $ \tanh^{-1} z $ is also valid for $ |z| < 1 $, so we have the same inequality condition $ |\zeta| < |\alpha| $. The previous equation gives the next two theorems. 
\begin{theorem} Let $ \zeta $ be a root of $ D_d z^2 + 2 \sqrt{\Delta} z + 1 = 0 $ that satisfies $ |\zeta| < \alpha $. Then
\begin{align*}
	\frac{B_{d - 1}}{2 \sqrt{\Delta}} \ln 3 = -\sum_{k = 0}^\infty \frac{B_{(2k + 1)d - 1}}{2k + 1} \zeta^{2k + 1}.
\end{align*}
\end{theorem}

\begin{proof} We choose $ \zeta $ such that 
\begin{align*}
	\frac{(\alpha - \beta) \zeta}{\alpha \beta - \zeta^2} = \frac{1}{2}.
\end{align*}
We obtain our theorem upon simplification and using the fact $ 2\tanh^{-1} (1/2) = \ln 3 $.
\end{proof}

\begin{theorem} Let $ \zeta $ be a root of $ D_d z^2 + 3 \sqrt{\Delta} z + 1 = 0 $ that satisfies $ |\zeta| < \alpha $. Then
\begin{align*}
	\frac{B_{d - 1}}{2 \sqrt{\Delta}} \ln 2 = -\sum_{k = 0}^\infty \frac{B_{(2k + 1)d - 1}}{2k + 1} \zeta^{2k + 1}.
\end{align*}
\end{theorem}

\begin{proof} We choose $ \zeta $ such that 
\begin{align*}
	\frac{(\alpha - \beta) \zeta}{\alpha \beta - \zeta^2} = \frac{1}{3}.
\end{align*}
We obtain our theorem upon simplification and using the fact $ 2\tanh^{-1} (1/3) = \ln 2 $.
\end{proof}

\subsubsection{Example (continued)}

For $ \sqrt{8} $, the first two theorems give
\begin{align*}
	\frac{\pi}{24 \sqrt{2}} &= \sum_{k = 0}^\infty \frac{(-1)^{k + 1} B_{4k + 1}}{2k + 1} \bigg( 2\sqrt{6} - 5 \bigg)^{2k + 1}, \\
	\frac{\pi}{32 \sqrt{2}} &= \sum_{k = 0}^\infty \frac{(-1)^{k + 1} B_{4k + 1}}{2k + 1} \bigg( \sqrt{23 + 16\sqrt{2}} - 2\sqrt{2} - 4 \bigg)^{2k + 1}.
\end{align*}
On the other hand, the latter two theorems give
\begin{align*}
	\frac{\ln 3}{8 \sqrt{2}} &= \sum_{k = 0}^\infty \frac{B_{4k + 1}}{2k + 1} \bigg( \sqrt{33} - 4\sqrt{2} \bigg)^{2k + 1}, \\
	\frac{\ln 2}{8 \sqrt{2}} &= \sum_{k = 0}^\infty \frac{B_{4k + 1}}{2k + 1} \bigg( \sqrt{73} - 6\sqrt{2} \bigg)^{2k + 1}.
\end{align*}

\subsection{Sums involving continuants}

The Binet formula enables us to calculate an intriguing sum of involving $ B_\nu $, as stated in the next theorem. 	
\begin{theorem} Let $ x \neq \alpha, \beta $. For integer $ r \geq -1 $, we have 
\begin{align*}
	\sum_{n = 1}^N x^n B_{nd + r} = \frac{x^{N + 1} (B_{(N + 1)d + r}/D_d + x B_{Nd + r}) - x(B_{d + r}/D_d + x B_r)}{(x - \alpha)(x - \beta)}.
\end{align*}
\end{theorem}

\begin{proof} It is straightforward to check that
\begin{align*}
	\sum_{n = 1}^N (x D_d)^n \frac{\alpha^n - \beta^n}{\alpha - \beta} &= \frac{x(xD_d)^{N + 1} (\alpha^N - \beta^N) + (xD_d)^{N + 1} (\alpha^{N + 1} - \beta^{N + 1}) - xD_d (\alpha - \beta)}{(\alpha - \beta) (x^2 D_d - xC_d - 1)}, \\	
	\sum_{n = 0}^{N - 1} (x D_d)^n \frac{\alpha^n - \beta^n}{\alpha - \beta} &= \frac{x(xD_d)^{N + 1} (\alpha^{N - 1} - \beta^{N - 1}) + (xD_d)^N (\alpha^N - \beta^N) - xD_d (\alpha - \beta)}{(\alpha - \beta) (x^2 D_d - xC_d - 1)}.
\end{align*}
Hence, by applying Binet's formula (Theorem \ref{Th:Binet_Formula}), we get
\begin{align*}
	\sum_{n = 1}^N x^n B_{nd + r} &= \sum_{n = 1}^N \bigg((-D_d)^{-1} (-x D_d)^n \frac{\alpha^n - \beta^n}{\alpha - \beta} B_{d + r} - x (-x D_d)^{n - 1} \frac{\alpha^{n - 1} - \beta^{n - 1}}{\alpha - \beta} B_r \bigg) \\
	&= \frac{-x^{N + 1} B_{(N + 1)d + r} - x^{N + 2} D_d B_{Nd + r} + xB_{d + r} + x^2 D_d B_r}{1 + xC_d - x^2 D_d} \\
	&= \frac{x^{N + 1} (B_{(N + 1)d + r}/D_d + x B_{Nd + r}) - x(B_{d + r}/D_d + x B_r)}{(x - \alpha)(x - \beta)}.
\end{align*}
\end{proof}

\begin{remark} One could also apply telescoping sum directly to the identity
\begin{align*}
	(x - \alpha)(x - \alpha) B_{kd + r} = x^{k + 1} (B_{(k + 1)d + r}/D_d + xB_{kd +r}) - x^k (B_{kd + r}/D_d + xB_{(k - 1)d +r}),
\end{align*}
although our derivation gives an explanation for this identity.
\end{remark}

For the next series, we take $ r = -1 $. 
\begin{theorem} It holds that
\begin{align*}
	\sum_{n = 1}^N \binom{N}{n} (-1)^n (\alpha + \beta)^n B_{nd - 1} =  \frac{B_{2Nd - 1}}{D_d^N}.
\end{align*}
\end{theorem}

\begin{proof} Note that $ 1 - C_d z = Dz^2 $ for $ z = \alpha $ and $ \beta $. Hence,
\begin{align*}
	\sum_{n = 1}^N \binom{N}{n} (-C_d)^n \frac{\alpha^n - \beta^n}{\alpha - \beta} = \frac{(1 - C_d \alpha)^N - (1 - C_d \beta)^N}{\alpha - \beta} = D_d^N \frac{\alpha^{2N} - \beta^{2N}}{\alpha - \beta}.
\end{align*}
The sum in question is then
\begin{align*}
	\sum_{n = 1}^N \binom{N}{n} (-1)^n (\alpha + \beta)^n B_{nd - 1} &= \sum_{n = 1}^N \binom{N}{n} \bigg( \frac{C_d}{D_d} \bigg)^n B_{nd - 1} \\
	&= -D_d^{N - 1} \frac{\alpha^{2N} - \beta^{2N}}{\alpha - \beta} B_{d - 1} = \frac{B_{2Nd - 1}}{D_d^N}.
\end{align*}
\end{proof}

\subsubsection{Example (continued)}

Returning to the example of $ \sqrt{8} $, we get, for $ x \neq 3 \pm 2\sqrt{2} $, the sums
\begin{align*}
	\sum_{n = 1}^N x^n B_{2n + r} = \frac{x^{N + 1} (-B_{2(N + 1) + r} + x B_{2N + r}) - x(-B_{2 + r} + x B_r)}{(x - 3 + 2\sqrt{2})(x - 3 - 2\sqrt{2})},
\end{align*}
and
\begin{align*}
	\sum_{n = 1}^N \binom{N}{n} (-6)^n B_{2n - 1} = (-1)^N B_{4N - 1}.
\end{align*}

\subsection{Divisibility properties} \label{Sec:Divisibility_Properties}

In this section, we explore some of the divisibility behaviors of the continuants. For similar results associated with Lucas sequences, refer to Ribenboim's book \cite{Ribenboim_04}.

\subsubsection{Divisibility within the sequence} \label{Sec:Divisibility_within_Sequence}

We saw in the proof of Theorem \ref{Th:Single_Reccurence} that $ B_{d - 1} \mid B_{2d - 1} $. This can be further extended to the following theorem.
\begin{theorem} If $ m \mid n $, then $ B_{md - 1} \mid B_{nd - 1} $.	
\end{theorem}
\begin{proof} Using Corollary \ref{Th:Catalan_Identity}, when $ \nu = nd - 1 $, $ \lambda = md $, then 
\begin{align*}
	B_{(m + n)d - 1} &= B_d B_{nd - 1} + a_1 B_{md - 1} B_{md - 2, 1}.
\end{align*}
We get our conclusion by an inductive argument, which is omitted here. 
\end{proof}
Stated differently, this theorem implies that the sequence $ (B_n') = (B_{nd - 1}) $ is a divisibility sequence. In fact, it is a strong divisibility sequence, i.e., $ B_{\gcd (m, n)}' = \gcd (B_m', B_n') $. One can prove this using the same technique for Fibonacci sequence, so we leave the proof as an exercise.

\subsubsection{Divisibility by primes} \label{Sec:Divisibility_by_Primes}

Let $ p $ be an odd prime. If $ p \mid C_d $ and $ p \mid D_d $, then $ p \mid B_{nd + r} $ (this also holds for $ 2 $). For the case of $ p \mid C_d $ and $ p \nmid D_d $, the result is given in the next theorem.
\begin{theorem} If $ p \mid C_d $, $ p \nmid D_d $, and $ n \geq 2 $, then
\begin{align*}
	B_{nd + r} &= \begin{cases} (-1/2)^{n - 2} D_d \Delta^{(n - 2)/2} B_r, \quad &\text{ if } n \text{ is even}; \\ (-1/2)^{n - 1} \Delta^{(n - 1)/2} B_{d + r}, \quad &\text{ if } n \text{ is odd}, \end{cases} \pmod{p}.
\end{align*}
In particular, $ B_{nd - 1} \equiv 0 \pmod{p} $ when $ n $ is even.
\end{theorem}

\begin{proof}
We expand the term $ (\alpha^n - \beta^n)/(\alpha - \beta) $ in Binet's formula (Theorem \ref{Th:Binet_Formula}),
\begin{align*}
	\frac{\alpha^n - \beta^n}{\alpha - \beta} &= \frac{1}{(2D_d)^{n - 1}} \sum_{k = 0}^{\lfloor (n - 1)/2 \rfloor} \binom{n}{2k + 1} (-C_d)^{n - 2k - 1} \Delta^k.
\end{align*}
If $ p \mid C_d $ and $ p \nmid D_d $, then
\begin{align*}
	(-D_d)^{n - 1} \frac{\alpha^n - \beta^n}{\alpha - \beta} &= \begin{cases} 0, \quad &\text{ if } n \text{ is even}; \\ \Delta^{(n - 1)/2}/2^{n - 1}, \quad &\text{ if } n \text{ is odd}, \end{cases} \pmod{p}.
\end{align*}
Hence, the theorem follows.
\end{proof}

If $ p \nmid C_d $ and $ p \mid D_d $, then $ \Delta \equiv C_d^2 $ (mod $ p $). Using this fact, we conclude the following theorem.
\begin{theorem} If $ p \nmid C_d $ and $ p \mid D_d $, then
\begin{align*}
	B_{(p - 1)d + r} \equiv B_r, \quad B_{(p - 1)d - 1} \equiv 0 \pmod{p}
\end{align*}
\end{theorem}

\begin{proof} When $ p \mid D_d $, the term involving $ B_r $ vanishes modulo $ p $. Consider again the expansion of $ (\alpha^n - \beta^n)/(\alpha - \beta) $, 
\begin{align*}
	2^{p - 1} B_{pd + r} \equiv \sum_{k = 0}^{(p - 1)/2} \binom{p}{2k + 1} (-C_d)^{p - 1} B_{d + r} \pmod{p}.
\end{align*}
Since $ \binom{p}{k} \equiv 0 $ (mod $ p $) for $ 1 \leq k \leq p - 1 $, this simplifies to 
\begin{align*}
	B_{pd + r} \equiv B_{d + r} \pmod{p},
\end{align*}
and the theorem follows.
\end{proof}

Suppose now that $ p \nmid C_d D_d $. By setting $ n = p + 1 $ in the expansion of $ (\alpha^n - \beta^n)/(\alpha - \beta) $, we get
\begin{align*}
	2^p B_{(p + 1)d + r} &= \sum_{k = 0}^{(p - 1)/2} \binom{p + 1}{2k + 1} C_d^{p - 2k} \Delta^k B_{d + r} + 2D_d \sum_{k = 0}^{(p - 1)/2} \binom{p}{2k + 1} C_d^{p - 2k - 1} \Delta^k B_r.
\end{align*}
Taking modulo $ p $ on both sides,
\begin{align*}
	2B_{(p + 1)d + r} \equiv C_d B_{d + r} + C_d \Delta^{(p - 1)/2} B_{d + r} + 2D_d \Delta^{(p - 1)/2} B_r \pmod{p}.
\end{align*}
If $ p \mid \Delta $, then
\begin{align*}
	2B_{(p + 1)d + r} \equiv C_d B_{d + r} \pmod{p}.
\end{align*}
This is equivalent to our next theorem.
\begin{theorem} \label{Th:Divisibility_Complete} If $ p \mid \Delta $ and $ p \nmid C_d D_d $, then
\begin{align*}
	2B_{pd + r} \equiv C_d B_r, \quad B_{pd - 1} \equiv 0\pmod{p}.
\end{align*}
\end{theorem}

Otherwise, $ \Delta^{(p - 1)/2} \equiv (\Delta \mid p) $ (mod $ p $), where $ (\cdot \mid p) $ denotes the Legendre symbol modulo $ p $. Hence, we have our next result. 
\begin{theorem} \label{Th:Divisibility_p+1} For an odd prime $ p \nmid C_d D_d $, if $ \Delta $ is a quadratic residue modulo $ p $, then
\begin{align*}
	B_{(p + 1)d + r} \equiv C_d B_{d + r} + D_d B_r, \quad B_{(p + 1)d - 1} \equiv C_d B_{d - 1} \pmod{p}.
\end{align*}	
If $ \Delta $ is a quadratic non-residue modulo $ p $, then
\begin{align*}
	B_{(p + 1)d + r} \equiv -D_d B_r, \quad B_{(p + 1)d - 1} \equiv 0 \pmod{p}.
\end{align*}
\end{theorem}

On the other hand, if we take $ n = p $ with $ p \nmid C_d D_d $, then
\begin{align*}
	2^{p - 1} B_{pd + r} &= \sum_{k = 0}^{(p - 1)/2} \binom{p}{2k + 1} C_d^{p - 2k - 1} \Delta^k B_{d + r} + 2D_d \sum_{k = 0}^{(p - 3)/2} \binom{p - 1}{2k + 1} C_d^{p - 2k - 2} \Delta^k B_r.
\end{align*}
Taking modulo $ p $ on both sides and using the fact that $ \binom{p - 1}{k} \equiv (-1)^k $ (mod $ p $), we get
\begin{align*}
	B_{pd + r} &\equiv \Delta^{(p - 1)/2} B_{d + r} - \frac{2D_d}{C_d} \sum_{k = 0}^{(p - 3)/2} \bigg( 1 + \frac{4D_d}{C_d^2} \bigg)^k B_r \pmod{p} \\
	&\equiv \Delta^{(p - 1)/2} B_{d + r} + \frac{C_d}{2} (1 - \Delta^{(p - 1)/2}) B_r \pmod{p}.
\end{align*}
Thus, we can deduce the following theorem.
\begin{theorem} \label{Th:Divisibility_p} For an odd prime $ p \nmid C_d D_d $, if $ \Delta $ is a quadratic residue modulo $ p $, then
\begin{align*}
	B_{(p - 1)d + r} &\equiv B_r, \quad B_{(p - 1)d - 1} \equiv 0 \pmod{p}.
\end{align*}
If $ \Delta $ is a quadratic non-residue modulo $ p $, then
\begin{align*}
	B_{pd + r} &\equiv C_d B_r - B_{d + r}, \quad B_{pd - 1} \equiv -B_{d - 1} \pmod{p}.
\end{align*}
\end{theorem}

\begin{remark} In particular, if $ d = 1 $, $ C_d = (a + 1) $ and $ D_d = -a $ for some integer $ a $, then $ \Delta = (a + 1)^2 - 4a = (a - 1)^2 $. For any prime $ p \nmid C_d D_d \Delta $, then $ (\Delta \mid p) = 1 $ and
\begin{align*}
	p \mid B_{p - 2} = a^{p - 2} \frac{a^{-(p - 1)} - 1}{a^{-1} - 1} = \frac{a^{p - 1} - 1}{a - 1}.
\end{align*}
Therefore, $ p \mid a^{p - 1} - 1 $ (this is trivial for $ p \mid C_d D_d \Delta $), which is Fermat's little theorem.
\end{remark}

Theorems \ref{Th:Divisibility_Complete} to \ref{Th:Divisibility_p} suggest a certain periodicity of $ B_\nu $ modulo $ p $. This period is commonly referred to as the Pisano period modulo $ p $. While we do not delve into the full theory here, we will demonstrate a minor consequence of these theorems. Set $ \sigma_p (n) $ as the order of $ n $ modulo $ p $, i.e., the smallest positive integer $ k $ such that $ n^k \equiv 1 \pmod p $. The next corollary gives the Pisano period modulo $ p $.
\begin{corollary} The Pisano period of $ (B_\nu) $ modulo $ p $, where $ p \nmid C_d D_d $, divides 
\begin{align*}
	\begin{cases} pd \sigma_p (C_d/2), &\quad \text{ if } p \mid \Delta; \\ (p - 1)d, &\quad \text{ if } \Delta \text{ is a quadratic residue modulo } p; \\ (p + 1)d \sigma_p (-D_d), &\quad \text{ if } \Delta \text{ is a quadratic non-residue modulo } p. \end{cases}
\end{align*}
\end{corollary}

\begin{proof}
If $ p \mid \Delta $, then Theorem \ref{Th:Divisibility_Complete} implies that 
\begin{align*}
	B_{pd \sigma_p (C_d/2) + r} \equiv (C_d/2)^{\sigma_p (C_d/2)} B_r \equiv B_r \pmod{p}.
\end{align*}
Thus, the Pisano period divides $ pd \sigma_p (C_d/2) $. The other cases are similar.
\end{proof}

\begin{remark} For a more detailed account of Pisano period involving Fibonacci numbers, refer to the papers by Falc\'{o}n \cite{Falcon_Plaza_09} and Wall \cite{Wall_60}.
\end{remark}

\subsubsection{Example (continued)}

Using the example of $ \sqrt{8} $ as before, we calculate $ \Delta = 32 $. We choose $ p = 3 $ and $ 7 $, since $ 32 $ is a quadratic non-residue modulo $ 3 $ and a quadratic residue modulo $ 7 $. Theorems \ref{Th:Divisibility_p+1} and \ref{Th:Divisibility_p} give us the following results, respectively:
\begin{align*}
	B_{r + 8} \equiv B_r \pmod{3}, \quad B_{r + 16} \equiv 6B_{r + 2} - B_r \pmod{7},
\end{align*}
and 
\begin{align*}
	B_{r + 6} \equiv 6B_r - B_{r + 2} \pmod{3}, \quad B_{r + 12} \equiv B_r \pmod{7}.
\end{align*}

\subsection{Law of apparition and repetition} \label{Sec:Law_of_Apparition_and_Repetition}

\subsubsection{Law of apparition}

The results in \S \ref{Sec:Divisibility_by_Primes} resemble the law of apparition of Lucas sequences. For any prime $ p $, define $ \omega(p) $ as the rank of apparition of $ p $, which is the minimum positive integer $ k $ such that $ B_{kd - 1} \equiv 0 $ (mod $ p $). The next theorem summarizes what we know about $ \omega (p) $. 
\begin{theorem}[Law of apparition] \label{Th:Law_of_Apparition} For a prime $ p $, we have 
\begin{enumerate}
	\item[(i)] If $ p \mid C_d $ and $ p \mid D_d $, then $ \omega (p) = 1 $.
	
	\item[(ii)] When $ p = 2 $, the following results hold: 
	\begin{itemize}
		\item If $ 2 \nmid C_d $ and $ 2 \mid D_d $, then
		\begin{align*}
			\omega (2) = \begin{cases} \text{does not exist}, \quad &\text{ if } B_{d - 1} \equiv 1 \pmod{2}; \\ 1, \quad &\text{ if } B_{d - 1} \equiv 0 \pmod{2}.
			\end{cases}
		\end{align*}
	
		\item If $ 2 \mid C_d $ (or $ 2 \mid \Delta $) and $ 2 \nmid D_d $, then $ \omega(2) = 1 $ or $ 2 $.
	
	\item If $ 2 \nmid C_d $ and $ 2 \nmid D_d $, then $ \omega(2) = 1 $ or $ 3 $.		
	\end{itemize}

	\item[(iii)] When $ p $ is an odd prime, the following results hold: 
	\begin{itemize}
		\item If $ p \mid C_d $ and $ p \nmid D_d $, then $ \omega(p) = 1 $ or $ 2 $.
		
		\item If $ p \nmid C_d $ and $ p \mid D_d $, then $ \omega(p) \mid p - 1 $.
		
		\item If $ p \nmid C_d $, $ p \nmid D_d $, and $ p \mid \Delta $, then $ \omega(p) = 1 $ or $ p $.
		
		\item If $ p \nmid C_d $, $ p \nmid D_d $, and $ p \nmid \Delta $, then $ \omega(p) \mid p - (\Delta \mid p) $.	
	\end{itemize}
\end{enumerate}	
\end{theorem}

\begin{proof}
We have already dealt with the case of odd primes, so we can suppose $ p = 2 $. Using the fact $ B_{2d - 1} = C_d B_{d - 1} $, we see that if $ p \mid C_d $ (or $ p \mid \Delta $) and $ p \nmid D_d $, then $ B_{2d - 1} = 0 $. Otherwise, if $ p \nmid C_d $ and $ p \nmid D_d $, then 
\begin{align*}
	B_{3d - 1} \equiv B_{2d - 1} + B_{d - 1} \equiv 2B_{d - 1} \equiv 0 \pmod{2}.
\end{align*}
However, if $ p \nmid C_d $, but $ p \mid D_d $, then for $ n \geq 1 $,
\begin{align*}
	B_{(n + 1)d - 1} \equiv B_{nd - 1} \pmod{2}.
\end{align*}
An inductive argument shows that $ B_{nd - 1} \equiv B_{d - 1} $ (mod $ 2 $) for all $ n \geq 1 $. Thus, it is possible that $ \omega (2) $ does not exists if $ B_{d - 1} \equiv 1 $ (mod $ 2 $).
\end{proof}

\subsection{Lucas pseudoprime}

Theorem \ref{Th:Law_of_Apparition} provides an extension to the notion of Lucas pseudoprime. For any positive integer $ n $, by an abuse of notation, set $ \epsilon (n) = (\Delta \mid n) $ as the Jacobi symbol modulo $ n $. We say that $ n $ is a Lucas pseudoprime with respect to $ (B_\nu) $, if $ n \nmid C_d D_d \Delta $ and 
\begin{align*}
	B_{(n - \epsilon (n))d - 1} \equiv 0 \pmod{n}.
\end{align*}
If the congruence is false, we can conclude that $ n $ is composite. However, if the congruence holds, it does not necessarily imply that $ n $ is prime and hence the term `pseudoprime'. For example, we take $ (B_\nu) $ as the sequence of denominators of the convergent to $ \sqrt{8} $ again. Choose $ n = 35 $, then the Jacobi symbol is $ (32 \mid 35) = -1 $. On the other hand, 
\begin{align*}
	B_{71} = 641614773393652358999201580 \equiv 0 \pmod{35}.
\end{align*}
Hence, $ 35 $ is a pseudoprime with respect to $ (B_\nu) $.

\subsubsection{Law of repetition}

For the next part, we require an identity that can be traced back to Lagrange. Given any odd integer $ m $, it holds that
\begin{align*}
	X^m - Y^m &= \sum_{k = 0}^{(m - 1)/2} \frac{m}{k} \binom{m - k - 1}{k - 1} (XY)^k \bigg( X - Y \bigg)^{m - 2k},
\end{align*}
This gives us an expansion for the term $ B_{mnd - 1}/B_{d - 1} $ as
\begin{align*}
	\frac{B_{mnd - 1}}{B_{d - 1}} = \sum_{k = 0}^{(m - 1)/2} \frac{m}{k} \binom{m - k - 1}{k - 1} \Delta^{(m - 2k - 1)/2} (-D_d)^{nk} \bigg( (-D_d)^{n - 1} \frac{\alpha^n - \beta^n}{\alpha - \beta} \bigg)^{m - 2k}.
\end{align*}
We write $ p^e \mid \mid N $ if $ e $ is the highest power of $ p $ dividing $ N $. The next theorem gives the law of repetition.
\begin{theorem} [Law of repetition]  \label{Th:Law_of_Repetition} If $ p^e \mid \mid B_{nd - 1}/B_{d - 1} $ and $ p \nmid m $, then $ p^{e + f} \mid B_{p^f mnd - 1}/B_{d - 1} $. In addition, if $ p \nmid D_d $, then the power is exact.	
\end{theorem}

\begin{proof} By induction, it suffices to establish the theorem for $ f = 0 $ and $ 1 $. The previous identity gives
\begin{align*}
	\frac{B_{mnd - 1}}{B_{d - 1}} &\equiv m (-D_d)^{n(m - 1)/2} \frac{B_{nd - 1}}{B_{d - 1}} \pmod{(B_{nd - 1}/B_{d - 1})^3}, \\
	\frac{B_{pmnd - 1}}{B_{d - 1}} &\equiv pm (-D_d)^{n(pm - 1)/2} \frac{B_{nd - 1}}{B_{d - 1}} \pmod{(B_{nd - 1}/B_{d - 1})^3}.
\end{align*}
Hence, if $ p^e \mid \mid B_{nd - 1}/B_{d - 1} $, then $ p^e \mid B_{mnd - 1}/B_{d - 1} $ and $ p^{e + 1} \mid B_{pmnd - 1}/B_{d - 1} $. The second part is clear.
\end{proof}

\subsection{Non-integer sequences}

While all the theorems presented above assume that $ (a_\nu) $ and $ (b_\nu) $ are integer sequences, there is no reason not to extend them to real sequences as well. The only casualties are that $ C_d $, $ D_d $, and consequently $ B_\nu $ are not necessarily integers. However, most of our results still apply except those involving divisibility properties in \S \ref{Sec:Divisibility_Properties}. We demonstrate here an example. Consider the sequence
\begin{align*}
	B_\nu = \begin{cases} \sqrt{3} B_{\nu - 1} + B_{\nu - 2}, \quad &\text{ if } \nu \text{ is odd}; \\ B_{\nu - 1} - B_{\nu - 2}, \quad &\text{ if } \nu \text{ is even}. \end{cases}
\end{align*}
The first few values of $ B_\nu $ are
\begin{align*}
	B_1 = \sqrt{3}, \quad B_2 = \sqrt{3} - 1, \quad B_3 = 3, \quad B_4 = 4 - 2\sqrt{3}.	
\end{align*}
We compute $ \alpha = -(\sqrt{3} - \sqrt{7})/2  $ and $ \beta = -(\sqrt{3} + \sqrt{7})/2 $, so using Binet's formula (Theorem \ref{Th:Binet_Formula}), we have
\begin{align*}
	B_{4k + 1} = \sqrt{\frac{3}{7}} (\alpha^{2k - 1} - \beta^{2k - 1}  - \sqrt{3} (\alpha^{2k} - \beta^{2k})).
\end{align*}
Plugging this value into Theorem \ref{Th:Telescoping_arctan}, we obtain
\begin{align*}
	\sum_{n = 1}^\infty \tan^{-1} \frac{3}{B_{4k + 1}} = \tan^{-1} \frac{\sqrt{3}}{4} + \tan^{-1} \frac{\sqrt{3}}{19} + \tan^{-1} \frac{\sqrt{3}}{91} + \ldots = \tan^{-1} \frac{1}{\sqrt{3}} = \frac{\pi}{6}.
\end{align*} 

\section{Further discussion}

The results presented in this paper primarily focus on the denominator, $ B_{\nu, \lambda} $. However, it should be possible to give similar results involving the numerator, $ A_{\nu, \lambda} $. We leave that to the readers.

Apart from that, the Fibonacci numbers have a companion sequence called the Lucas numbers. By examining the initial values, we observe that $ (A_{\nu, \lambda}) $ is not the desired companion sequence. Therefore, it is still an open problem to construct the corresponding Lucas sequence for $ (B_{\nu, \lambda}) $.


\begin{thebibliography}{}	
	\bibitem{Andrica_Bagdasar_20} D. Andrica and O. Bagdasar, \textit{Recurrent Sequences}, Springer Nature, 2020.
	
	\bibitem{Benjamin_Su_Quinn_00} A. T. Benjamin, F. E. Su, and J. J. Quinn, Counting on continued fractions, \textit{Math. Mag.} \textbf{73} (2000), 98--104.
	
	\bibitem{Carson_07} T. R. Carson, Periodic recurrence relations and continued fractions, \textit{Fibonacci Quart.} \textbf{45} (2007), 357--361.
	
	\bibitem{Castellanos_86} D. Castellanos, Rapidly converging expansions with Fibonacci coefficients, \textit{Fibonacci Quart.} \textbf{24} (1986), 70--82.

	\bibitem{Castellanos_89} D. Castellanos, A generalization of Binet's formula and some of its consequences, \textit{Fibonacci Quart.} \textbf{27} (1989), 424--438.
	
	\bibitem{Edson_Lewis_Yayenie_11} M. Edson, S. Lewis, and O. Yayenie, The $ k $-Periodic Fibonacci sequence and an extended Binet's formula, \textit{Integers} \textbf{11} (2011), \#A32.
	
	\bibitem{Falcon_Plaza_09} S. Falc\'{o}n and \'{A}. Plaza, $ k $-Fibonacci sequence modulo $ m $, \textit{Chaos Solitons Fractals} \textbf{41} (2009), 497--504.	
	
	\bibitem{Frontczak_Prasad_23} R. Frontczak and K. Prasad, Balancing polynomials, Fibonacci numbers and some new series for $ \pi $, \textit{Mediterr. J. Math.}, \textbf{20} (2023), Article 207. 
	
	\bibitem{Hoggatt_Bruggles_64} V. E. Hoggatt Jr. and I. D. Bruggles, A primer on the Fibonacci sequence: Part V, \textit{Fibonacci Quart.} \textbf{2} (1964), 59--65.

	\bibitem{Khrushchev_08} S. Khrushchev, \textit{Orthogonal Polynomials and Continued Fractions}, Cambridge University Press, 2008.
	
	\bibitem{Lehmer_75} D. H. Lehmer, Fibonacci and related sequences in periodic tridiagonal matrices, \textit{Fibonacci Quart.} \textbf{13} (1975), 150--158.

	\bibitem{Lucas_78} E. Lucas, Th\'{e}orie des fonctions num\'{e}riques simplement p\'{e}riodiques, \textit{Amer. J. Math.} \textbf{1} (1878), 184--196; 197--240; 289--321.
	
	\bibitem{Millin_74} D. A. Millin, Advanced problems H-237, \textit{Fibonacci Quart.} \textbf{12} (1974), 309--312.
	
	\bibitem{Muir_60} T. Muir, \textit{A Treatise on the Theory of Determinants}, Dover Publications, 1960.
	
	\bibitem{Panario_Sahin_Wang_13} D. Panario, M. Sahin, and Q. Wang, A family of Fibonacci-like conditional sequences, \textit{Integers} \textbf{14} (2013), \#A78.	
	
	\bibitem{Perron_13} O. Perron, \textit{Die Lehre von den Kettenbr\"{u}chen}, Leipzig und Berlin, B.G. Teubner, 1913.
	
	\bibitem{Ribenboim_04} P. Ribenboim, \textit{The Little Book of Bigger Primes}, Springer New York, 2004.
	
	\bibitem{OEIS} N. J. A. Sloane, The On-Line Encyclopedia of Integer Sequences, \url{https://oeis.org}.		
	
	\bibitem{Tan_11} E. Tan, Some properties of the bi-periodic	Horadam sequences, \textit{Notes Number Theory Discrete Math.} \textbf{23} (2011), 56--65.
	
	\bibitem{Wall_60} D. D. Wall, Fibonacci series modulo $ m $, \textit{Amer. Math. Monthly} \textbf{67} (1960), 525--532.
	
	\bibitem{Yayenie_11} O. Yayenie, A note on generalized Fibonacci sequence, \textit{Appl. Math. Comput.} \textbf{217} (2011), 5603--5611.
\end{thebibliography}
\end{document}